\documentclass[10pt]{amsart}
\usepackage{amsmath}
\usepackage{amsthm}
\usepackage{amscd}
\usepackage{amssymb}
\usepackage{amsfonts}
\usepackage{graphics}
\usepackage[pdftex,colorlinks]{hyperref}
\date{}
\newlength{\defbaselineskip}
\newcommand{\setlinespacing}[1]
           {\setlength{\baselineskip}{#1 \defbaselineskip}}

\long\def\salta#1{\relax}

\theoremstyle{plain}
\newtheorem{theorem}{Theorem}[section]
\newtheorem{proposition}[theorem]{Proposition}
\newtheorem{lemma}[theorem]{Lemma}

\theoremstyle{definition}
\newtheorem{definition}[theorem]{Definition}

\newtheorem{remark}[theorem]{Remark}

\theoremstyle{remark}

\newcommand{\elle}[1]{L^{#1}(\Omega)}

\newcommand{\re}{\mathbb{R}}

\def\uo{u_{0}}
\def\utn{u^{n}_{\vare}}
\def\uu{\underline{u}}

\def\luo{L^{1}(\Omega)}
\def\lio{L^{\infty}(\Omega)}

\def\rn{\mathbb{R}^{N}}

\def\rife#1{(\ref{#1})}

\def\parl#1#2{L^{#1}(0,T;L^{#2}(\Omega))}

\def\luq{L^{1}(Q)}
\def\liq{L^{\infty}(Q)}
\def\ou{\overline{u}}
\def\vare{\varepsilon}

\def\t1p0{T^{1,p}_{0}(\Omega)}

\def\m2{M^{\frac{N(p-1)}{N-1}}(\Omega)}

\def\div{\mathrm{div}}

\def\tkv{T_{k}(v)}

\def\into{\int_{\Omega}}
\def\intq{\int_{Q}}

\def\w-1p'{W^{-1,p'}(\Omega)}
\def\pw-1p'u{L^{p'}(0,1;W^{-1,p'}(\Omega))}

\def\dys{\displaystyle}

\def\luo{L^{1}(\Omega)}
\def\lp'n{(L^{p'}(\Omega))^{N}}

\def\lio{L^{\infty}(\Omega)}
\def\liq{L^{\infty}(Q)}

\keywords{Asymptotic behavior, linear parabolic equations, measure data}

\author[F. Petitta]{Francesco Petitta}
\email{francesco.petitta@sbai.uniroma1.it}

\address[F. Petitta]{Dipartimento di Scienze di Base e Applicate
per l' Ingegneria, ``Sapienza", Universit\`a di Roma, Via Scarpa 16, 00161 Roma, Italy.}
\begin{document}

\setlinespacing{1}
\begin{abstract}
In this paper we deal with the asymptotic behavior as $t$ tends to infinity of solutions for linear parabolic equations  whose model is
$$
\begin{cases}
    u_{t}-\Delta u = \mu & \text{in}\ (0,T)\times\Omega,\\[0.7 ex]
    u(0,x)=u_0 & \text{in}\ \Omega,
  \end{cases}
$$
where  $\mu$ is a general, possibly singular, Radon measure which does not depend on time, and $u_0\in L^{1}(\Omega)$. We prove that the duality solution, which exists and is unique, converges to the duality solution (as introduced in \cite{s}) of the associated elliptic problem.
\end{abstract}
\title[Linear parabolic equations, Asymptotic behavior, measure data]{Asymptotic behavior of solutions for linear parabolic equations with general measure data} \maketitle
\setcounter{equation}{0}
\section{Introduction}\label{sec1}

A large number of papers has been devoted to the study of asymptotic behavior for solutions  of parabolic problems under various assumptions and in different contexts: for a review on classical  results see \cite{f}, \cite{a}, \cite{sp}, and references therein. More recently in \cite{pe} and \cite{lp} the case of  nonlinear monotone operators, and  quasilinear problems with nonlinear absorbing terms having natural growth, have been considered;  in particular, in \cite{pe}, we dealt with nonnegative measures $\mu$ absolutely continuous with respect to the parabolic $p$-capacity (the so called \emph{soft measures}). Here we analyze the case of linear operators with possibly singular general measures and no sign assumptions on the data.

Let $\Omega\subseteq \rn$ be a bounded open set, $N\geq 2$, $T>0$; we denote by $Q$  the cylinder $(0,T)\times\Omega$. We are interested in the  study of main properties and in the asymptotic behavior with respect
to the time variable $t$ of the solution of the linear parabolic
problem 
\begin{equation}\label{plin3}
\begin{cases}
u_t +L (u)=\mu& \text{in}\ (0,T)\times\Omega,\\
u(0)=u_0,& \text{in}\ \Omega,\\
u=0& \text{on}\ (0,T)\times\partial\Omega,
\end{cases}
\end{equation}
with $\mu\in \mathcal{M}(Q)$ the space of Radon measures with bounded total variation on $Q$, $u_0\in \luo$, and $$L(u)=-\div(M(x)\nabla u),$$
 where $M$  is a matrix with bounded, measurable entries, and satisfying the ellipticity assumption
\begin{equation}
\label{coercp}
M(x)\xi\cdot\xi\geq \alpha |\xi|^2,
\end{equation}
for any $\xi\in\rn$, with $\alpha >0$.

In order to obtain uniqueness, in the elliptic case, the notion of duality solution of Dirichlet problem 
 \begin{equation}\label{elin3}
\begin{cases}
  -\div{(M(x)\nabla v)} =\mu & \text{in}\ \Omega,\\[1.5 ex]
 v=0 &\text{on}\ \partial\Omega,
  \end{cases}
\end{equation}
was introduced in \cite{s}.

Following the idea of \cite{s} we can define a solution of problem \rife{plin3} in a duality sense as follows
\begin{definition}\label{dualdef}
A function $u\in\luq$ is a \emph{duality solution} of problem \rife{plin3} if 
\begin{equation}\label{dualsense}
-\into \uo w(0)\ dx+\intq u\, g\ dxdt=\intq w\ d\mu,
\end{equation}
for every $g\in\liq$, where $w$ is the solution of the \emph{retrograde problem}  
\begin{equation}\label{retroc1}
\begin{cases}
    -w_{t}- \mathrm{div} (M^{\ast}(t,x)\nabla w)
 =g & \text{in}\ (0,T)\times\Omega,\\
    w(T,x)=0 & \text{in}\ \Omega,\\
 w(t,x)=0 &\text{on}\ (0,T)\times\partial\Omega,
  \end{cases}
\end{equation}
where $M^{\ast}(t,x)$ is the transposed matrix of $M(t,x)$.
\end{definition}
\begin{remark}
Notice that all terms in \rife{dualsense} are well defined thanks to standard parabolic regularity results (see \cite{lsu}, \cite{e}). Moreover, it is quite easy to  check that any  duality solution of problem \rife{plin3} actually turns out to be a distributional solution of the same problem. Finally recall that any duality solution turns out to coincide with the renormalized solution of the same problem (see \cite{pe1}); this notion introduced in \cite{dmop} for the elliptic case, and then adapted to the parabolic case in \cite{pe1} should be the right one to ensure uniqueness  also in the nonlinear framework.
\end{remark}

A unique duality solution for problem \rife{plin3} exists, in fact we have the following 
\begin{theorem}\label{exiuni}
Let $\mu\in \mathcal{M}(Q)$ and $u_0\in\luo$, then there exists a unique duality solution of problem \rife{plin3}.
\end{theorem}

The main result of this paper concerns the asymptotic behavior of the duality solution of problem \rife{plin3}, in the case where  the measure $\mu$  do not depend on time. 

 First observe that by  Theorem \ref{exiuni} a unique solution 
 is well defined for all $t>0$. We are interested in the asymptotic behavior of $u(t,x)$ as $t$ tends to infinity. We recall that by a duality solution of problem \rife{elin3} we mean a function $v\in\luo$ such that  
 \begin{equation}\label{elldualsense}
\into v\, g\ dxdt=\into z\ d\mu,
\end{equation}
for every $g\in\lio$, where $z$ is the variational solution of the dual  problem  
\begin{equation}\label{dualell}
\begin{cases}
  -\mathrm{div} (M^{\ast}(x)\nabla z)
 =g & \text{in}\ \ \Omega,\\
   z(x)=0 &\text{on}\ \ \partial\Omega.
  \end{cases}
\end{equation}

As we will see later, a duality solution of problem \rife{plin3} turns out to be continuous with values in $\luo$. Let us state our main result:
\begin{theorem}\label{asi}
Let $\mu\in \mathcal{M}(Q)$  be independent on the variable $t$. Let $u(t,x)$ be the duality solution of problem
\rife{plin3} with $u_0 \in \luo$, and let $v(x)$ be the duality solution of the corresponding
elliptic problem \rife{elin3}. Then
\[
\lim_{T\rightarrow +\infty} u(T,x)=v(x),   
\]
in $L^{1}(\Omega)$.  
\end{theorem}

\setcounter{equation}{0}
\section{Existence and uniqueness of the duality solution}
Let us prove Theorem \ref{exiuni}:
\begin{proof}
Let us first prove the result in the case $\mu\in\luq$ and $u_0$ smooth; let us fix $r,q\in\re$ such that 
\[
r, \,q>1,\ \ \ \ \ \frac{N}{q}+\frac{2}{r}<2\,,
\]
and let us consider $g\in \parl{r}{q}\cap\liq$. Let $w$ be the solution of problem \rife{retroc1}; standard parabolic regularity results (see again \cite{lsu}) imply that $w$ is continuous on $Q$ and  
\[
\|w\|_{\liq}\leq C \|g\|_{\parl{r}{q}};
\]
therefore,  the linear functional 
\[
\Lambda:\parl{r}{q}\mapsto\re,
\]
defined by
\[
\Lambda(g)=\intq w\ d\mu + \into u_0w(0)\,,
\]
is well-defined and continuous, since
\[
\dys|\Lambda (g)|\leq (\|\mu\|_{\mathcal{M}(Q)}+\|u_0\|_{\lio})\|w\|_{\liq} \leq C \|g\|_{\parl{r}{q}}.
\]
So, by \emph{Riesz's representation theorem} there exists a unique $u\in\parl{r'}{q'}$ such that
\[
\Lambda(g)=\intq u\,g\ dxdt,
\]
for any $g\in\parl{r}{q}$.
So we have that, if $\mu\in\luq$ and $u_0$ is smooth, then there exists a (unique by construction) duality solution of problem \rife{plin3}.

 A standard approximation argument shows that a unique solution  also exists for problem \rife{plin3} if $\mu\in \mathcal{M}(Q)$ and $u_0\in \luo$. In fact, via a standard convolution argument, we can approximate $u_0$ in $\luo$ with smooth functions $u_0^\vare$, and $\mu$ with smooth functions $\mu^\vare$ in the narrow topology of measures, that is 
 $$
\lim_{\vare\to 0}\intq\varphi\ d\mu^\vare = \intq\varphi\ d\mu, \ \ \ \forall \ \varphi\in C(\overline{Q}),
$$
and $\|\mu^\vare\|_{\luq}\leq C$. Hence, reasoning as in the proof of  Theorem $1.2$ in  \cite{bdgo}, one can show that there exists a function $u\in\luq$ such that $u^\vare$ converges to $u$ in $\luq$ and so we can pass to the limit in the duality formulation of $u^\vare$ to obtain the result.
\end{proof}
\setcounter{equation}{0}

\section{Asymptotic behavior}

In this section we will prove Theorem \ref{asi}.
From now on we will denote by $T_k (s)$ the function $\text{max}(-k,\text{min}(k,s))$ and $\Theta_k (s)$ will indicate its primitive function, that is:
\[
\Theta_{k}(s)=\int_0^s T_k (\sigma)\ d\sigma .
\]
Let us prove the following preliminary result:
\begin{proposition}\label{vsol} Let $\mu\in \mathcal{M}(Q)$ be independent on time
and let $v$ be the duality solution of the elliptic problem 
 \begin{equation}\label{dualasiee}
\begin{cases}
-\div(M(x)\nabla v)=\mu& \text{in}\ \Omega,\\
v=0,&  \text{on}\ \partial \Omega.
\end{cases}
\end{equation}
Then $v$ is the unique solution of the parabolic problem
   \begin{equation}\label{dualasiep}
\begin{cases}
w_t -\div(M(x)\nabla w)=\mu& \text{in}\ (0,T)\times\Omega,\\
w(0)=v(x),&  \text{in}\  \Omega,\\
w(t,x)=0 & \text{on}\ (0,T)\times\partial\Omega,
\end{cases}
\end{equation}
in the duality sense introduced in Definition  \ref{dualdef}, for any fixed $T>0$.
\end{proposition} 
\begin{proof}
We have to check that $v$ is a solution of problem \rife{dualasiep}; to do that let us choose  $\tkv$ as test function in \rife{retroc1}. We obtain
$$
\begin{array}{l}
\quad\dys-\int_0^T\langle w_t,\tkv\rangle \ dt+\intq M^\ast (x) \nabla w\cdot\nabla \tkv\ dxdt
\dys=\intq \tkv\ g\ dxdt.
\end{array}
$$
Now, integrating by parts we have
\[
-\int_0^T\langle w_t,\tkv\rangle \ dt=\into w(0) v(x) +\omega(k),
\]
where $\omega(k)$ denotes a nonnegative quantity which vanishes as $k$ diverges, while 
\[
\intq \tkv\ g\ dxdt= \intq v\,g \ dxdt +\omega(k).
\]
Finally, using Theorem $2.33$ and Theorem $10.1$ of \cite{dmop}, we have
\[
\intq M^\ast(x) \nabla w\cdot\nabla \tkv\ dxdt=\intq M(x)\nabla\tkv\cdot\nabla w\ dxdt=\int_0^T \into w \ d\lambda_k (x)\, dt,
\]
where the $\lambda_k $ are measures in $\mathcal{M}(\Omega)$ which converge to $\mu$ in the narrow topology of measures;  thus, recalling that $w$ is bounded continuous, and using the dominated convergence theorem, we have  
\[
\intq M^\ast (x) \nabla w\cdot\nabla \tkv\ dxdt=\intq w\ d\mu +\omega(k).
\]
  Gathering together all these facts, we have that $v$ is a duality solution of \rife{plin3} having itself as initial datum. 
\end{proof}
Proposition \ref{vsol} allows us to deduce that the duality solution of problem \rife{plin3} $u$ belongs to $C(0,T;\luo)$ for any fixed $T>0$; indeed, $z=u- v$ uniquely solves  problem
\begin{equation}\label{dualasiz}
\begin{cases}
z_t -\div(M(x)\nabla z)=0& \text{in}\ (0,T)\times\Omega,\\
z(0)=u_0 - v & \text{in}\ \Omega,\\
z=0 & \text{on}\ (0,T)\times\partial\Omega,
\end{cases}
\end{equation}
in the duality sense, and so $z\in C(0,T;\luo)$. This is due to a result of \cite{po}, since $z$ turns out to be an entropy solution in the sense of the definition given in \cite{p}.

So, we have  that $u$ satisfies 
\begin{equation}\label{dualdual}
\intq u\, g\ dxdt=\intq w\ d\mu+\into u_0 \, w(0)\ dx,
\end{equation}
for any $g\in\liq$, where $w$ is the unique solution of the retrograde problem
\begin{equation}\label{retasi}
\begin{cases}
    -w_{t}- \mathrm{div} (M^\ast (x)\nabla w)
 =g & \text{in}\ (0,T)\times\Omega,\\
    w(T,x)=0 & \text{in}\ \Omega,\\
 w(t,x)=0 &\text{on}\ (0,T)\times\partial\Omega.
  \end{cases}
\end{equation}

Therefore, as we said before, for fixed $\mu$ and $g\in\liq$  one can uniquely determine $u$ and $w$,   solution of the above problems, defined for any time $T>0$.

Moreover, let us give the following definition:
\begin{definition}
A function $u\in \luq$ is a \emph{duality supersolution} of problem \rife{plin3} if 
\[
\intq u \, g\ dxdt \geq \intq w\ d\mu +\into u_0 w(0)\ dx,
\]
for any bounded $g\geq 0$, and $w$ solution of \rife{retasi}, while $u$ is a \emph{duality subsolution} if $-u$ is a duality supersolution.
\end{definition}
\begin{lemma}\label{asilemma}
Let $\ou$ and $\uu$ be respectively a duality supersolution and a duality subsolution for problem \rife{plin3}. Then
 $\uu\leq\ou$.
 \end{lemma}
 \begin{proof}
Simply subtract the formulations for $\underline{u}$ and $\overline{u}$ to obtain
 \[
 \intq (\uu-\ou)g\ dxdt\leq 0, 
 \]
 for any $g\geq 0$, and so $\uu\leq\ou$. 
 \end{proof}
 \begin{remark}\label{perognit}
Observe that, if the functions in Lemma \ref{asilemma} are continuous with values in $\luo$, then we actually have that $\uu(t,x)\leq\ou(t,x)$ for every fixed $t$, a.e on $\Omega$. 
 \end{remark}
 \begin{proof}[Proof of Theorem \ref{asi}] 
   We split the proof in few steps. 
   
   {\it Step $1$.}  Let us  first suppose $u_0 =0$ and $\mu\geq 0$. 
   If  we consider a parameter $s>0$ we have that both $u(t,x)$ and $u_{s}(t,x)\equiv u(t+s,x)$ are duality 
 solutions of problem \rife{plin3} with, respectively, $0$ and $u(s,x)\geq 0$ as initial datum;
 so,  from Lemma \ref{asilemma} we deduce that $u(t+s,x)\geq u(t,x)$ for $t,s>0$. Therefore $u$ is a monotone nondecreasing
 function in $t$ and so it converges to a function  $\tilde{v}(x)$ almost everywhere and  in $\luo$ since, thanks to Proposition \ref{vsol} and Lemma \ref{asilemma}, $u(t,x)\leq v(x)$. \\

Now, recalling that $u$ is obtained as limit of regular solutions with smooth data $\mu_\vare$, we can define $\utn (t,x)$ as the solution of
\begin{equation}\label{dualasitn}
\begin{cases}
(\utn)_t -\div(M(x)\nabla \utn)=\mu_\vare & \text{in}\ (0,1)\times\Omega,\\
\utn(0,x)=u_\vare (n,x) & \text{in}\ \Omega\\
\utn= 0 & \text{on}\ (0,1)\times\partial\Omega.
\end{cases}
\end{equation}

On the other hand, if $g\geq 0$, we define $w^n (t,x)$ as 
\begin{equation}\label{retasin}
\begin{cases}
    -w^{n}_{t}- \mathrm{div} (M^\ast (x)\nabla w^n)
 =g & \text{in}\ (0,1)\times\Omega,\\
    w^n (1,x)=w(n+1,x) & \text{in}\ \Omega,\\
 w^n =0 &\text{on}\ (0,1)\times\partial\Omega.
  \end{cases}
\end{equation}
Recall that, through the change of variable $s=T-t$,  $w$ solves a similar linear parabolic problem, so that if $g\geq 0$, by classical comparison results one has that $w(t,x)$  is decreasing in time. 
  Moreover, by comparison principle, we have that $w^n $ is increasing with respect to $n$ and,  again by comparison
 Lemma \ref{asilemma}, we have that, for fixed $t\in (0,1)$
\[
w^{n}(1,x)\leq w^{n}(t,x)=w(n+t,x)\leq w(n,x)=w^{n-1}(1,x),
\]
and so
 its limit $\tilde{w}$ does not depend on time and is the solution of
  \begin{equation}\label{retasie}
\begin{cases}
    - \mathrm{div} (M^\ast (x)\nabla \tilde{w})
 =g & \text{in}\ \Omega,\\
 \tilde{w}(x)=0 &\text{on}\ \partial\Omega .
  \end{cases}
\end{equation}
An analogous argument shows that also the limit of $u^n$ (which exists thanks to standard compactness arguments, see for instance \cite{bdgo} again) does not depend on time.
   Thus, using $\utn $ in  \rife{retasin} and $w^n$ in \rife{dualasitn}, integrating by parts,  subtracting, and passing to the limit over $\vare$,  we obtain
  \[
 \int_0^1 \into u^n \, g -\int_0^1 \into w^n\ d\mu +\into u^n (0) w^n (0)\ dx -\into u^n (1) w^n (1)\ dx=0. 
  \]
  Hence, we can pass to the limit on $n$ using monotone convergence theorem obtaining
   \begin{equation}\label{dacita}
  \into \tilde{v} \, g - \into \tilde{w}\ d\mu \ dx=0, 
  \end{equation}
  and so $v=\tilde{v}$.

  If $g$ has no sign we can reason separately with $g^+$ and $g^-$ 
  obtaining \rife{dacita} and then using the linearity of \rife{dualdual} to conclude.

 If $v$ is the duality  solution of problem \rife{elin3}, we proved in Proposition \ref{vsol} that $v$ is also the duality solution
of the initial boundary value problem \rife{plin3}
with $v$ itself as initial datum. Therefore, by comparison Lemma \ref{asilemma}, if $0\leq u_{0}\leq v$, we have that
the solution $u(t,x)$ of \rife{plin3} converges to $v$ in $\elle{1}$ as $t$ tends to infinity; in fact, we proved it for the duality solution with homogeneous initial datum, while $v$ is a nonnegative duality solution with itself as initial datum.

  {\it Step $2$.} Now, let us take $u_\lambda (t,x)$ the solution of problem \rife{plin3} with  $u_{0}= \lambda v$ as initial datum
  for some $\lambda>1$ and again $\mu\geq 0$.
Hence, since $\lambda v$ does not depend on  time, we have that it is a duality supersolution of the parabolic problem \rife{plin3},
 and, observing that $v$ is a subsolution of the same problem,
we can apply again the comparison lemma  finding that $v(x)\leq u_\lambda (t,x)\leq
\lambda v(x) $ a.e. in $\Omega$, for all positive $t$.

Moreover, thanks to the fact that the datum $\mu$ does not depend on time, we can apply the comparison result also between
$u_\lambda (t+s,x)$ solution with $u_{0}=u_\lambda (s,x)$, with $s$ a positive
parameter, and $u_\lambda (t,x)$, the solution with $u_{0}=\lambda v$ as initial datum; so we obtain $u_\lambda (t+s,x)\leq u_\lambda (t,x)$ for
all $t,s>0$, a.e. in $\Omega$. So, by virtue of this monotonicity
result we have that there exists a function $\overline{v}\geq v$ such
that $u_\lambda (t,x)$ converges to $\overline{v}$ a.e. in $\Omega$ as $t$
tends to infinity. Clearly  $\overline{v}$ does not depend on
$t$ and we can develop the same argument used before 
 to prove that we can pass to the limit in the approximating duality 
formulation, and so, by uniqueness, we can obtain that
$\overline{v}=v$. So, we have proved that the result holds for
the solution starting from $u_{0}= \lambda v $ as initial
datum, with $\lambda >1$ and $\mu\geq 0$. Since we proved before that the result holds true also for the solution starting from $u_{0}=0$, then,
again applying a comparison argument, we can conclude in the same way that the convergence to $v$ holds true for solutions
 starting from $u_{0}$ such that
 $0\leq u_{0}\leq \lambda v$ as initial datum, for fixed $\lambda>1$.

 {\it Step $3$.} Now, let $u_0  \in\luo$ a nonnegative function and $\mu\geq 0$,  and  recall that, thanks to suitable Harnack inequality (see \cite{t}), if $\mu\neq 0$, then $v>0$ (which implies $\lambda v$ tends to $+\infty$ on $\Omega$ as $\lambda$ diverges). Without loss of generality we can suppose $\mu\neq 0$ (the case $\mu\equiv 0$ is the easier one and it can be proved as in \cite{pe});
let us define the monotone nondecreasing
  (with respect to $\lambda$) family of functions
\[
u_{0,\lambda}=\min(u_{0},\lambda v).   \]

  As we have shown above, for every fixed $\lambda>1$,  $u_{\lambda}(t,x)$, the duality solution of problem
  \rife{plin3} with $u_{0,\lambda}$ as initial datum, converges to $v$ a.e. in $\Omega$, as $t$ tends to infinity. Moreover, using again standard compactness arguments,we also have
   that $T_{k}(u_{\lambda}(t,x))$ converges to $T_{k}(v)$ weakly
 in $H^1_0 (\Omega)$ as $t$ diverges, for every fixed $k>0$.

So, thanks to \emph{Lebesgue theorem}, 
 we  can easily check that $u_{0,\lambda}$ converges to $u_{0}$ in $L^{1}(\Omega)$ as
 $\lambda$ tends to infinity.
  Therefore, using a stability result for renormalized solutions of the linear problem \rife{plin3}
 (see \cite{pe1}) we obtain that  $T_{k}(u_{\lambda}(t,x))$ converges to $T_{k}(u(t,x))$
 strongly in $L^2 (0,T;H^1_0 (\Omega))$ as $\lambda$  tends to infinity.

On the other hand, since $z_\lambda=u-u_\lambda$ solves the problem
\begin{equation}
\begin{cases}
(z_\lambda)_t -\div(M(x)\nabla z_\lambda)=0& \text{in}\ (0,T)\times\Omega,\\
z_\lambda (0)=u_0-u_{0,\lambda} & \text{in}\ \Omega,\\
z_\lambda =0 & \text{on}\ (0,T)\times\partial\Omega,
\end{cases}
\end{equation}
in the duality sense, then $z_\lambda$  turns out to be an entropy solution of the same problem and so we have (see \cite{p})
\[
\into \Theta_{k}(u-u_{\lambda})(t)\ dx\leq
\into\Theta_{k}(u_{0}-u_{0,\lambda})\ dx,
\]
for every $k,\  t>0$. Dividing the above inequality by
$k$, and passing to the limit as $k$ tends to $0$ we obtain
\begin{equation}\label{unif}
\|u(t,x)-u_{\lambda}(t,x)\|_{L^{1}(\Omega)}\leq
\|u_{0}(x)-u_{0,\lambda}(x)\|_{L^{1}(\Omega)},
\end{equation}
for every $t>0$. Hence, we have
\[
\|u(t,x)-v(x)\|_{L^{1}(\Omega)}\leq
\|u(t,x)-u_{\lambda}(t,x)\|_{L^{1}(\Omega)}+\|u_{\lambda}(t,x)-v(x)\|_{L^{1}(\Omega)};
\]
then, thanks to the fact that the estimate in (\ref{unif}) is
uniform in $t$, for every fixed $\epsilon$, we can choose
$\bar{\lambda}$ large enough such that
\[
\|u(t,x)-u_{\bar{\lambda}}(t,x)\|_{L^{1}(\Omega)}\leq
\frac{\epsilon}{2},
\]
for every $t>0$; on the other hand, thanks to the result proved
above, there exists $\bar{t}$ such that
\[
\|u_{\bar{\lambda}}(t,x)-v(x)\|_{L^{1}(\Omega)}\leq \frac{\epsilon}{2},
\]
for every $t>\bar{t}$, and this concludes the proof of the result in the case of nonnegative data $\mu$ and $u_0\in\luo$.

{\it Step $4$.} Let $\mu\in \mathcal{M}(Q)$ be independent on $t$ and $u_0 \in\luo$ with no sign assumptions. We consider again  the function $z(t,x)= u(t,x)-v(x)$; thanks to Proposition \ref{vsol} it turns out to solve problem
\begin{equation}\label{dualasize}
\begin{cases}
z_t -\div(M(x)\nabla z)=0& \text{in}\ (0,T)\times\Omega,\\
z(0)=u_0 - v & \text{in}\ \Omega,\\
z=0 & \text{on}\ (0,T)\times\partial\Omega,
\end{cases}
\end{equation}
and so, if either $u_0 \leq v$ or $u_0\geq v$ then the result is true since $z(t,x)$ tends to zero in $\luo$ as $t$ diverges thanks to what we proved above. Now,  if $u^\oplus$ solves
$$
\begin{cases}
u^{\oplus}_{t} -\div(M(x)\nabla u^\oplus)=\mu & \text{in}\ (0,T)\times\Omega,\\
u^\oplus (0)=\max{(u_0 , v)} & \text{in}\ \Omega,\\
u^\oplus=0 & \text{on}\ (0,T)\times\partial\Omega,
\end{cases}
$$
and $u^\ominus$ solves
$$
\begin{cases}
u^{\ominus}_{t} -\div(M(x)\nabla u^\ominus)=\mu & \text{in}\ (0,T)\times\Omega,\\
u^\ominus (0)=\min{(u_0 , v)} & \text{in}\ \Omega,\\
u^\ominus=0 & \text{on}\ (0,T)\times\partial\Omega,
\end{cases}
$$
 then by comparison we have $u^\ominus (t,x)\leq u(t,x)\leq u^\oplus (t,x)$ for any $t$, a. e. in $\Omega$, and this concludes the proof since the result holds true for both $u^\oplus$ and $u^\ominus$.
 \end{proof}


\begin{thebibliography}{10}
\addcontentsline{toc}{chapter}{Bibliografia}
\bibitem{a} A. Arosio,  Asymptotic behavior as $t\rightarrow +\infty $ of solutions of linear parabolic equations with discontinuous coefficients in a bounded domain, Comm. Partial Differential Equations \textbf{4} (1979), no. 7, 769--794. 
\bibitem{bdgo} L. Boccardo, A.  Dall'Aglio,  T. Gallou\"et, L. Orsina,
 Nonlinear parabolic equations with measure data, J. Funct. Anal., \textbf{147} (1997) no.1,  237--258.
 
\bibitem{dmop} G. Dal Maso, F.  Murat, L.  Orsina, A.  Prignet, Renormalized
 solutions of elliptic equations with general measure
data,\, Ann. Scuola Norm. Sup. Pisa Cl. Sci., \textbf{28} (1999), 741--808.
\bibitem{e}  L.C. Evans,  \emph{Partial Differential Equations}, A.M.S., 1998. 
 \bibitem{f} A. Friedman,
\emph{Partial differential equations of parabolic type}, Englewood Cliffs, N.J., Prentice-Hall 1964. 
\bibitem{lsu} O. A. Ladyzhenskaja, V. Solonnikov,N. N. Uraltceva,  {\it Linear and quasilinear parabolic
equations},  Academic Press,  (1970).
\bibitem{lp} T. Leonori, F. Petitta,  Asymptotic behavior of solutions for parabolic equations with natural growth term and irregular data, Asymptotic
Analysis \textbf{48(3)} (2006), 219--233.\bibitem{pe} F. Petitta, Asymptotic behavior of solutions for
parabolic operators of Leray-Lions type and measure data, Advances in  Differential Equations \textbf{12} (2007), no. 8, 867--891.
\bibitem{pe1} F. Petitta,   Renormalized solutions of nonlinear parabolic equations with general measure data, Ann. Mat. Pura ed Appl.,  \textbf{187} (4)  (2008),  563--604.
\bibitem{po} A. Porretta, Existence results for nonlinear parabolic equations via strong convergence of truncations,
 Ann. Mat. Pura ed Appl. (IV), \textbf{177} (1999), 143--172.
 \bibitem{p} A. Prignet, Existence and uniqueness of entropy solutions of parabolic problems with $L^{1}$ data, Nonlin. Anal. TMA \textbf{28} (1997), 1943--1954.
\bibitem{sp} S. Spagnolo, Convergence de solutions d'\'equations d'\'evolution,   Proceedings of the International Meeting on Recent Methods in Nonlinear Analysis (Rome, 1978),  311--327, Pitagora, Bologna, 1979. 
\bibitem{s} G. Stampacchia, Le probl\`eme de Dirichlet pour les \'equations elliptiques du seconde ordre \`a coefficientes discontinus, Ann. Inst. Fourier (Grenoble), \textbf{15} (1965), 189--258.
\bibitem{t} N. S. Trudinger, On Harnack type inequalities and their application to quasilinear elliptic equations, Comm. Pure Appl Math, \textbf{20} (1967), 721--747.







\salta{
\bibitem{bm} Barles G., Murat F.,
\emph{ Uniqueness and maximum principle for quasilinear elliptic
equations with quadratic growth conditions}, Arch. Rational Mech.
Anal.\textbf{ 133} (1995) n. 1, 77--101.
\bibitem{bgl} Benachour  S., Karch G.,
Lauren\c{c}ot P., \emph{Asymptotic profiles of solutions to viscous Hamilton-Jacobi equations}, 
J. Math. Pures Appl. (9) \textbf{83} (2004), no. 10, 1275--1308. 



\bibitem{bp} Blanchard D., Porretta A., \emph{Nonlinear parabolic equations with natural growth terms and measure initial data},  Ann. Scuola Norm. Sup. Pisa Cl. Sci. \textbf{ 30} (2001), no. 3-4.
\bibitem{bgo2} Boccardo L., Gallou\"et T., Orsina L., \emph{Existence and nonexistence of solutions for some nonlinear elliptic equations}, J. an. math. \textbf{73} (1997), 203--223.
\bibitem{do} Dall'Aglio A., Orsina L., \emph{Existence results for some nonlinear parabolic equations with nonregular data}, Differential Integral Equations, \textbf{5} (1992), 1335--1354.
\bibitem{do2} Dall'Aglio A., Orsina L., \emph{Nonlinear parabolic equations with natural growth conditions and $ L^1$ data}, Nonlinear Anal. T.M.A., \textbf{27} 1  (1996), 59--73.
\bibitem{f} Friedman A.,
Partial differential equations of parabolic type, Englewood Cliffs, N.J., Prentice-Hall 1964. 
\bibitem{g} Grenon N., \emph{Asymptotic behaviour for some quasilinear parabolic equations},  Nonlinear Anal.  \textbf{20}  (1993),  no. 7, 755--766. 
\bibitem{dib} E. DiBenedetto, Partial Differential Equations, Birkh\"{a}user, Boston, 1995.


\bibitem{po} Porretta A., \emph{Existence results for nonlinear parabolic equations via strong convergence of truncations}, 
 Ann. Mat. Pura ed Appl. (IV), \textbf{177} (1999), 143--172.
 \bibitem{sp} Spagnolo S., \emph{Convergence de solutions d'\'equations d'\'evolution},   Proceedings of the International Meeting on Recent Methods in Nonlinear Analysis (Rome, 1978),  311--327, Pitagora, Bologna, 1979. 
 }
\end{thebibliography}
\end{document}